\documentclass[11pt]{article}
\usepackage[toc,title,page]{appendix}
\usepackage[usenames]{color}
\usepackage{amssymb,amsmath,latexsym}
\numberwithin{equation}{section}
\usepackage{amsthm}
\usepackage[active]{srcltx}
\usepackage[mathscr]{eucal}
\usepackage{setspace}
\usepackage[color=yellow,textsize=small]{todonotes}
\usepackage{geometry}
\usepackage{multirow}
\usepackage{graphicx}
\usepackage{float}
\usepackage{amsfonts}
\usepackage{amssymb}
\usepackage{amsthm}
\usepackage{newlfont}
\usepackage{stmaryrd}
\usepackage{mathrsfs}
\usepackage[colorlinks,
            linkcolor=blue,
            anchorcolor=blue,
            citecolor=blue
            ]{hyperref}
\usepackage{indentfirst}
\usepackage{cases}
\usepackage{bookmark}
\usepackage{graphicx}
\usepackage{pdfpages}
\usepackage{setspace}
\usepackage{abstract}
\usepackage{cite}

\newsavebox{\mybox}

\allowdisplaybreaks[4]

\newtheorem{theorem}{Theorem}[section]
\newtheorem{lemma}[theorem]{Lemma}
\newtheorem{corollary}[theorem]{Corollary}
\newtheorem{definition}[theorem]{Definition}

\newtheorem{proposition}[theorem]{Proposition}
\newtheorem{condition}[theorem]{Condition}

\newtheorem{remark}[theorem]{Remark}

\def\dbE{{\mathbb{E}}}
\def\dbF{{\mathbb{F}}}

\def\dbP{{\mathbb{P}}}

\def\f{\varphi}

\def\o{\omega}

\def\O{\Omega}

\def\cF{{\cal F}}

\def\${|\!|\!|}
\def\({\left(}
\def\){\right)}

\newcounter{bean}
\newcommand{\benuma}{\setlength{\labelwidth}{.25in}
	\begin{list}%
		{(\alph{bean})}{\usecounter{bean}}}
	\newcommand{\eenuma}{\end{list}}

\newcommand{\bi}{\begin{itemize}}
	\newcommand{\ei}{\end{itemize}}
\newcommand{\be}{\begin{enumerate}}
	\newcommand{\ee}{\end{enumerate}}
\newcommand{\beqs}{\begin{equation*}}
\newcommand{\eeqs}{\end{equation*}}
\newcommand{\beq}{\begin{equation}}
    \newcommand{\eeq}{\end{equation}}
    \newcommand{\bald}{\begin{aligned}}
        \newcommand{\eald}{\end{aligned}}
\newcommand{\beqys}{\begin{eqnarray*}}
	\newcommand{\eeqys}{\end{eqnarray*}}
\newcommand{\beqy}{\begin{eqnarray}}
\newcommand{\eeqy}{\end{eqnarray}}

\mathchardef\mhyphen="2D
\newcommand{\bp}{\begin{pmatrix}}
	\newcommand{\ep}{\end{pmatrix}}
\begin{document}
\title{ \bf  Infinite dimensional open-loop linear quadratic stochastic optimal control problems  and related  games}
\date{}

\author{Guangdong Jing \thanks{  
    {E-mail:} 
   {jingguangdong@mail.sdu.edu.cn.} }
\\  \\  
School of Mathematics and Statistics, \\ 
Beijing Institute of Technology, 
Beijing 100081, China.   
}

\maketitle
\begin{abstract}
  \noindent
We investigate the linear quadratic stochastic optimal control problems in infinite dimension without Markovian restriction for coefficients. The necessary and sufficient conditions for open-loop optimal controls are presented. We prove the Fr\'echet differentiable of the cost functional with respect to the control variable, and the Fr\'echet derivatives are characterized in detail by operators derived from dual analysis, which are proven to be the stationary conditions. Transposition methods are adopted to deal with the adjoint equations. As applications, we employ the results to study open-loop Nash equilibria for two-person stochastic differential games. 

\end{abstract}
 
\bigskip


\bigskip


\section{Introduction}

For any given initial pair $(t,\eta) \in [0,T)\times H$, consider a control system governed by the following linear stochastic evolution equations (SEEs, for short) on a finite time horizon:
\begin{equation}\label{zfxk}
    \left\{
    \begin{aligned}
& d x(s)=[(A+A_1(s)) x(s)+B(s) u(s)+b(s)] d s 
\\  & \indent \indent 
+[C(s) x(s)+D(s) u(s)+\sigma (s)] d W(s) \quad \text{in}\ (t,T],
\\
& x(t)=\eta,      
\end{aligned}
\right.
\end{equation}
where $A$ generates a $C_0$-semigroup $\{e^{At}\}_{t\ge0}$ on $H$.  
The process $u(\cdot) \in \cal U[t,T]:= L_{\mathbb{F}}^{2}(t, T ; U) $  denotes the control variable, while ${ x(\cdot) }$  denotes the state process.

Associated with the control system \eqref{zfxk}, consider the following quadratic cost functional:
\begin{equation}\label{bdjfs}
    \begin{aligned}
        \mathcal{J}(t,\eta ; u(\cdot)) = & \frac{1}{2} \mathbb{E}\Big[ \int_{t}^{T} \big( \langle Q(s) x(s), x(s)\rangle_{H}+\langle R(s) u(s), u(s)\rangle_{U} + 2 \langle S(s)x(s),u(s) \rangle_U 
        \\  & \indent 
        +2 \langle \mathfrak{q}(s), x(s)\rangle_{H} 
        +2 \langle \mathfrak{r} (s), u(s)\rangle_{U} \big) d s 
         +\langle G x(T), x(T)\rangle_{H} + 2 \langle \mathfrak{g}, x(T)\rangle_{H}\Big].  
    \end{aligned}
\end{equation}


The stochastic linear quadratic (LQ, for short) optimal control problems are formulated as follows. 

\noindent 
\textbf{Problem (SLQ).} For each initial pair $(t,\eta) \in [0,T)\times H$, find a $ {u}^*(\cdot) \in L_{\mathbb{F}}^{2}(t, T ; U)$, such that 
\begin{align}\label{mdfzx}
\mathcal{J}(t, \eta ;  {u}^*(\cdot))=\inf _{u(\cdot) \in L_{\mathbb{F}}^{2}(t, T ; U)} \mathcal{J}(t, \eta ; u(\cdot)). 
\end{align}

For any $(t,\eta) \in [0,T)\times H$ and  $u(\cdot) \in L_{\mathbb{F}}^{2}(0, T ; U)$, under certain conditions (cf. \eqref{H1}, by \hyperref[mtuee]{Lemma \ref{mtuee}}), there uniquely exists a mild solution $x(\cdot) \equiv x(\cdot ;t, \eta, u,(b,\sigma )) \in C_{\mathbb{F}}([t, T] ; L^{2}(\Omega ; H))$, such that 
\begin{align}\label{fgz}
|x(\cdot)|_{C_{\mathbb{F}} ([t, T] ; L^{2}(\Omega ; H) )} \leq \mathcal{C} (|\eta|_{H}+|u(\cdot)|_{L_{\mathbb{F}}^{2}(t, T ; U)} +|b|_{L_{\mathbb F}^2(\O;L^1(t,T;H))} +|\sigma |_{L_{\mathbb F}^2(t,T;H)}). 
\end{align}
Hence with further conditions \eqref{H2}, \eqref{bdjfs} also becomes well-defined.

We introduce the following common concepts in control theory. 
\begin{definition}\label{dlkwf}
    1) The Problem (SLQ) is said to be a standard ${L Q}$ problem, if for $a.e. (t,\o)$, 
    \[
    Q(\cdot) \geq 0,\quad   R \gg 0, \quad  G \geq 0 ;
    \]
    
    2) The Problem (SLQ) is called finite at ${\eta \in H}$, if the right hand side of \eqref{mdfzx}  is finite;
    
    3) The Problem (SLQ) is called (uniquely) solvable at ${\eta \in H}$, if there (uniquely) exists a control ${ {u}^*(\cdot) \in L_{\mathbb{F}}^{2}(0, T ; U)}$ satisfying \eqref{mdfzx}. If so, ${ {u}^*(\cdot)}$ is named an (the) optimal control, while the corresponding ${ {x}^*(\cdot)}$ are called an (the) optimal state, ${( {x}^*(\cdot),  {u}^*(\cdot))}$ an (the) optimal pair, respectively;
    
    4) The Problem (SLQ) is called finite (resp. (uniquely) solvable), if it is finite (resp. (uniquely) solvable) for any ${\eta \in H}$.
\end{definition}

The LQ stochastic optimal control problem is the structurally simplest nontrivial model in optimal stochastic control, which can be used to approximate various complex nonlinear models in biology, finance, and other fields. The study on LQ optimal control problem has its origins at least before Bellman et al. \cite{BGG58}. And numerous works have been presented since then, such as Bensoussan et al. \cite{BDDM07} for deterministic control systems. 
On the other hand, it is currently well-known that the coefficients of one system are difficult to measure accurately, and it is also challenging to explicitly express all factors with  variables in physical laws. A good approach to overcome these problems is to introduce randomness to blur the effects caused by these uncertainties, and to transforme it into a stochastic model for research. From this point of view, it is meaningful to add stochastic factors into the deterministic LQ models. 
See Chen et al.\cite{CLZ98}, Sun-Yong \cite{sunyong20b1,sunyong19b1} for stochastic control systems in finite dimension, and the rich references therein. 
Besides, infinite dimension setting obviously covers more situations than the finite dimension counterpart, but with far more technical difficulties (cf. Kotelenez\cite{Ko08}). As a consequence of its importance, there are a great many of works on the optimal control problems for SEEs. For example,  Liu-Tang \cite{LT23} with newly defined operator valued conditional expectation, they regard the operator valued stochastic integral as error terms and absorb it making use of limit analysis to deal with  stochastic integration of operator valued stochastic processes. Besides, the transpositon solutions for BSEEs for both vector valued and operator valued processes serving as the first and second order adjoint processes separately for optimal control problems, can be found in detail in the monograph Lü-Zhang\cite{luqzhx21}. And especially we refer several works in LQ setting (cf. Guatteri-Tessitore\cite{GG14}, Hafizoglu et al.\cite{HLLM17}, Lü\cite{L19}, Lü-Wang\cite{LW23}, Xue et al. \cite{XXZ24}). 

The expression of open-loop optimal control involves systems in the form of forward and backward SEEs (FBSEEs, for short) \eqref{aapm}. It is natural to further study its well-posedness.
However, solving the FBSEEs \eqref{aapm} is quite tough. Recently, Xu et al. \cite{xtm24} studied the  domination-monotonicity conditions associated with this kinds of systems,  employing the method of continuation. The study in our paper also provides another background for their work. As for its finite dimensional counterpart, namely forward and backward stochastic differential equations, has attracted extensive attentions for more than three decades with lots of related monographs have appeared (cf. \cite{zhjf17}, and the references therein). Nevertheless, there are rather few conclusion for FBSEEs due to its difficulties, such as the less of It\^o's formula in infinite dimension. It deserves much more attention in the future. Besides, Addona et al. \cite{AMP23} studied the uniqueness of solution for semilinear stochastic Euler-Bernoulli beam equations which describe elastic systems with structural damping, and they utilized properties of certain related FBSEEs but not through the classical It\^o-Tanaka trick.
Both \cite{AMP23} and \cite{xtm24}, as well as others, have to appeal to Yosida approximation and finite dimension approximation, but the convergence analysis is difficult due to the lack of compactness arguments in infinite setting. Interested readers are referred to \cite{luqzhx21} for more detailed introduction. \cite{XXZ24} utilized an ``discretization then continuousization" method to cope with the infinite dimensional nature of PDE systems.

It was shown by Chen et al.\cite{CLZ98} that the LQ problems may be solvable even with indefinite control weight costs, which is not true in deterministic counterpart. Seemingly, stochastic factors give us more chances to solve the control problems. Nevertheless, Lü et al.\cite{LWZ17} showed that solvable stochastic LQ problems may NOT have feedback control. Even though we always prefer seeking feedback operators in the face of LQ problems due to its robustness and elegant explicit form, the facts in \cite{LWZ17} tell us that sometimes we have no choice but to identify the optimal control utilizing the open-loop necessary conditions. This justifies the value of the kind of results in this paper to some extent. 
Even so, we acknowledge that the feedback operator is indeed one of the charming aspects of LQ problems, but its complexity is beyond the scope of this paper and will be presented elsewhere.
And as shown by \cite{sunyong19b1}, the conclusions in open-loop optimal control for stochastic LQ problems are important stepstones for the study of its optimal feedback operators.

The main contributions of this paper are summarized as follows. Utilizing dual analysis methods and by introducing certain adjoint equations, we obtain the Fr\'echet differentiable of the cost functional with respect to the control variable, and characterize the Fr\'echet derivatives in detail with the help of the operators derived from the dual analysis. Moreover, the necessary and sufficient conditions for a control to be open-loop optimal are proven to be equivalent to the convexity (or equivalently, nonnegative) of certain cost functional of related homogeneous control problems and a stationary condition (which is exactly the above Fr\'echet derivatives) together with one coupled forward and backward system. The nonhomogeneous features of the models allow one more step to study stochastic differential games. Besides, we do not make the usual Markovian assumption on the coefficients and weight operators.

The rest of this paper is organized as follows. In \hyperref[sna]{Section \ref{sna}}, we present necessary notations and technical conditions for the coefficients and weight operators. Some elementary materials about SEEs and BSEEs are gethered in \hyperref[slaf]{Section \ref{slaf}}.  In \hyperref[sbsd]{Section \ref{sbsd}}, we perform the dual analysis and character the Fr\'echet differentiable of the cost functional. The necesary and sufficient conditions for a control to be open-loop optimal are studied in \hyperref[smos]{Section \ref{smos}}. \hyperref[sygs]{Section \ref{sygs}} is dedicated to the application of the main results to two-person stochastic differential games. At last, we make some conclusions and discussions in \hyperref[slcd]{Section \ref{slcd}}.

\section{Notations and assumptions}\label{sna}
\subsection{Notations}
Denote $(\O, \cF, \mathbf F, \dbP)$ a complete filtered probability space,  $\mathbf F=\{\cF_t\}_{t\ge0}$ the natural filtration, and define a one-dimensional standard Brownian motion $\{W(t)\}_{t\ge0}$ on it.  Denote the progressive $\sigma$-field  corresponding to $\mathbf F$ by $\dbF$.  
Denote $H$ (and $U$) a real separable Hilbert space, whose norm $|\cdot|_{H}$ and inner product $ \langle \cdot,\cdot \rangle_{H}$ are defined. Assume that $A$ is an unbounded linear operator on $H$, generating a $C_0$-semigroup $\{e^{At}\}_{t\geq 0}$. 
Denote by $A^*$ the adjoint operator of $A$.  $D(A)$ is a Hilbert space with the usual graph norm, and $A^*$ is the infinitesimal generator of $\{e^{A^*t}\}_{t\geq 0}$, the adjoint $C_0$-semigroup of $\{e^{At}\}_{t\geq 0}$. 

Denote $X$ a Banach space, with the norm ${|\cdot|_{X}}$. Denote $L_{\mathcal{F}_{t}}^{p}(\Omega ; X)$ the Banach space of all ${\mathcal{F}_{t}}$-measurable random variables $\xi: \Omega \rightarrow X$ satisfying  $\mathbb{E}|\xi|_{X}^{p}<\infty$, ${t \in[0, T]}$, $p \in[1, \infty)$.  
Besides, denote $D_{\mathbb{F}} ([0, T] ; L^{p}(\Omega ; X) )$ the vector space of all $X$-valued $\mathbf{F}$-adapted processes $\varphi(\cdot):[0, T] \rightarrow L_{\mathcal{F}_{T}}^{p}(\Omega ; X)$, which is right continuous with left limits and equipped with the norm 
$|\varphi(\cdot)|_{D_{\mathbb{F}}([0, T] ; L^{p}(\Omega ; X))}:= \sup _{t \in[0, T)}(\mathbb{E}|\varphi(t)|_{X}^{p})^{1 / p}$. 
Denote $C_{\mathbb{F}}([0, T] ; L^{p}(\Omega ; X))$ the Banach space of all the $X$-valued $\mathbf{F}$-adapted processes $\varphi(\cdot):[0, T] \rightarrow L_{\mathcal{F}_{T}}^{p}(\Omega ; X)$ which is continuous and whose norm is inherited from $D_{\mathbb{F}}([0, T] ; L^{p}(\Omega ; X))$. 
What's more, denote two Banach spaces 
with $p_{1}, p_{2}, p_{3}, p_{4} \in[1, \infty)$
\begin{align*}
L_{\mathbb{F}}^{p_{1}} (\Omega ; L^{p_{2}}(0, T ; X) )=  \bigg\{&f:(0, T) \times \Omega \rightarrow X \  |\   f(\cdot)  \ \text{is }  \mathbf{F}\mbox{-}\text{adapted and}
\\
&|f|_{L_{\mathbb{F}}^{p_{1}}\left(\Omega ; L^{p_{2}}(0, T ; X)\right)} :=  \left[\mathbb{E}\left(\int_{0}^{T}|f(t)|_{X}^{p_{2}} d t\right)^{\frac{p_{1}}{p_{2}}}\right]^{\frac{1}{p_{1}}}<\infty \bigg\}
\end{align*}
and 
\begin{align*}
L_{\mathbb{F}}^{p_{2}}\left(0, T ; L^{p_{1}}(\Omega ; X)\right)=  \bigg\{&f:(0, T) \times \Omega \rightarrow X \  |\   f(\cdot)  \ \text{is }  \mathbf{F}\mbox{-}\text{adapted and}
\\ 
&|f|_{L_{\mathbb{F}}^{p_{2}}\left(0, T ; L^{p_{1}}(\Omega ; X)\right)} := \left[\int_{0}^{T}\left(\mathbb{E}|f(t)|_{X}^{p_{1}}\right)^{\frac{p_{2}}{p_{1}}} d t\right]^{\frac{1}{p_{2}}}<\infty \bigg \}. 
\end{align*}
If $p_1=p_2=p$, denote the above two spaces as $L_{\mathbb{F}}^{p } (0, T ;  X)$.  
For another Banach space $Y$, denote $\cal L(X,Y)$ the space of bounded linear operators from $X$ to $Y$, which are equipped with the usual operator norm, and, if $X=Y$, suppressed as $\cal L(X)$.
$\mathcal{S}(H)$ (resp. $\mathcal{S}(U)$) denotes the set of all the bounded self-adjoint operators in the Hilbert space $H$ (resp. $U$).

\subsection{Assumptions}
In this subsection, we present the conditions assumed in this paper. 
\begin{condition}
    We impose the following conditions for the coefficient operators and the nonhomogeneous terms in the Eq. \eqref{zfxk}:
\begin{equation}\tag{H1}\label{H1}
\begin{aligned}
    &   A_1(\cdot) \in L_{\mathbb{F}}^{1} (0, T ; L^\infty(\O;\mathcal{L}(H))), \quad  B(\cdot) \in L_{\mathbb{F}}^\infty(\O; L^{2} (0, T ; \mathcal{L}(U ; H))), 
       \\
    &   C(\cdot) \in L_{\mathbb{F}}^{2} (0, T ; L^\infty(\O;\mathcal{L}(H))),  \quad  D(\cdot) \in L_{\mathbb{F}}^{\infty} (0, T ; \mathcal{L} (U ;H ) ), 
    \\ & 
    b(\cdot) \in L_{\mathbb F}^2 (\O; L^1 (0,T; H)), \quad   \sigma(\cdot)\in L_{\mathbb F}^2 (0,T; H).
\end{aligned}
\end{equation}
\end{condition}
\begin{condition}
    Assume the following conditions for the weighting operators in \eqref{bdjfs}: 
\begin{equation}\label{H2}\tag{H2}
    \begin{aligned}
    & Q(\cdot) \in L_{\mathbb{F}}^\infty(\O; L^{2}(0, T ; \mathcal{S}(H))), \quad  R(\cdot) \in L_{\mathbb{F}}^{\infty}(0, T ; \mathcal{S}(U)), 
    \\  &   
    S(\cdot) \in L_{\mathbb{F}}^\infty(\O; L^{2} (0, T ; \mathcal{L}(H; U))),  \quad  G \in  L_{\mathcal{F}_T}^{\infty}(\O;\mathcal{S}(H)), 
    \\  & 
    \mathfrak q \in  L_{\mathbb F}^2(\O;L^1(t,T;H)), \quad  \mathfrak r \in L_{\mathbb{F}}^{2}(t, T ; U),   \quad  \mathfrak g \in L_{\mathcal{F}_T}^{2}(\O; H).
\end{aligned}
\end{equation}
\end{condition}
\begin{remark}
    If we assume $B(\cdot) \in L_{\mathbb{F}}^\infty(\O; L^{2} (0, T ; \mathcal{L}(U ; H)))$ as above, the constant $\cal C$ in \eqref{fgz} is dependent on $|B(\cdot)|_{L_{\mathbb{F}}^\infty(\O; L^{2} (0, T ; \mathcal{L}(U ; H)))}$. If it strengthens as $B(\cdot) \in L_{\mathbb{F}}^\infty (0, T ; \mathcal{L}(U ; H))$, the dependence will disappear. 
    However, the coefficient $D(\cdot)$ can not be relaxed in general, since we expect the same integrability between the control $u(\cdot)\in L_{\mathbb{F}}^{2}(t, T ; U)$ and non-homogeneous term $\sigma(\cdot) \in L_{\mathbb{F}}^{2}(t, T ; H)$ under current well-posedness results in literature, unless imposing higher order regularity restrictions for the control variable. 
\end{remark}

\section{Preliminary materials}\label{slaf}
\subsection{Stochastic evolution equation}
Consider the following $H$-valued SEE:
\begin{equation}\label{eqmaror}
    \left\{
\begin{aligned}
&d X(t)=(AX(t)+f(t,X(t)))dt +\widetilde f(t,X(t))dW(t) \quad \mbox{in}\  (0,T] 
\\
&X(0)=X_0, 
\end{aligned}
\right.
\end{equation}
where $X_0: \O \to H$ is an $\cF_0$-measurable random variable, $A$ generates a $C_0$-semigroup $\{S(t)\}_{t\ge0}$ on $H$, and $f(\cdot, \cdot), \widetilde f(\cdot, \cdot): [0,T]\times \O \times H \to H$ are two given  functions satisfying the following conditions:
\begin{condition}\label{eqmarorllccl}
(i) Both $f(\cdot,h)$ and $\widetilde f(\cdot, h)$ are $\mathbf F$-adapted for any given $h\in H$;

(ii) There exist two nonnegative functions $L_1(\cdot)\in L^1(0,T), L_2(\cdot)\in L^2(0,T)$, such that for any given $h_1, h_2 \in H$ and a.e. $t\in[0,T]$, 
\begin{equation}\label{eqmarorllc}
    \left\{
    \begin{aligned}
       & |f(t,h_1)-f(t,h_2)|_H\le L_1(t)|h_1-h_2|_H, 
        \\
       & |\widetilde f(t,h_1)-\widetilde f(t,h_2)|_H\le L_2(t)|h_1-h_2|_H,
    \end{aligned}
    \right.
    \indent   \dbP\mbox{-}a.s. 
\end{equation}
\end{condition}
\begin{definition}
    An $H$-valued $\mathbf F$-adapted continuous stochastic process $X(\cdot)$ is called a mild solution to \eqref{eqmaror} if $f(\cdot, X(\cdot))\in L^1(0,T;H)\ a.s.$, $\widetilde f(\cdot, X(\cdot))\in L_{\mathbb F}^{2,loc}(0,T;H)$, and for any $t\in[0,T]$, 
    \begin{equation}
        \begin{aligned}
            X(t)=S(t)X_0+\int_0^tS(t-s)f(s,X(s))ds +\int_0^t S(t-s)\widetilde f(s,X(s))dW(s) \quad \dbP\mbox{-}a.s. 
        \end{aligned}
    \end{equation} 
\end{definition}
\begin{lemma}\label{mtuee}
    Let Condition \ref{eqmarorllccl}   hold and $f(\cdot, 0)\in L_{\mathbb F}^p(\O;L^1(0,T;H))$, $\widetilde f(\cdot, 0)\in L_{\mathbb F}^{ p}(\O;L^2(0,T;H))$ for some $p\ge2$. Then for any $X_0\in L_{\cal F_0}^{ p}(\O;H)$, the above Eq. \eqref{eqmaror} admits a unique mild solution $X(\cdot)\in C_{\mathbb F}([0,T];L^p(\O;H))$. Moreover, 
\begin{align}\label{mtueer}
    |X(\cdot)|_{C_{\mathbb F}([0,T];L^{ p}(\O;H))} \le C\big(|X_0|_{L_{\cal F_0}^{ p}(\O;H)} +|f(\cdot, 0)|_{L_{\mathbb F}^p(\O;L^1(0,T;H))} +|\widetilde f(\cdot, 0)|_{L_{\mathbb F}^p(\O;L^2(0,T;H))}\big). 
\end{align}
\end{lemma}
The  above \hyperref[mtuee]{Lemma \ref{mtuee}} can be found in \cite[Theorem 3.14]{luqzhx21}.

\subsection{Backward stochastic evolution equation}

Consider the following  $H$-valued backward SEE (BSEE, for short):
\begin{equation}\label{apHm}
\left\{
\begin{aligned} 
&dy(t) = - [ A^* y(t) - f(t,y(t),Y(t))]dt + Y(t) dW(t) &\mbox{ in }[0,T),\\
& y(T) = y_T.
\end{aligned}
\right.
\end{equation}
Here, $f:[0,T]\times  \O \times H\times H  \to H$,  $y_T \in L_{\cal F_T}^{p}( \O;H)$,  $p\in(1,2]$. In detail, we assume the following conditions. 
\begin{condition}\label{iclb}
    The map $f:[0,T]\times  \O \times H\times H  \to H$ in \eqref{apHm} satisfies 
    
    (i) For any given $h_1,h_2\in H$, $f(\cdot,h_1,h_2)$ is $\mathbf F$-adapted;
        
    (ii) There exist two nonnegative functions $L_1(\cdot)\in L^1(0,T), L_2(\cdot)\in L^2(0,T)$, such that for any given $h_1,h_2,\tilde h_1,\tilde h_2\in H $ and a.e. $t\in[0,T]$, 
    \begin{equation}\label{iclbe}
    \begin{aligned}
        |f(t,h_1,h_2)-f(t,\tilde h_1,\tilde h_2)|_H\leq  L_1( t)|h_1-\tilde h_1|_H+L_2( t)|h_2-\tilde h_2|_H  
    \end{aligned}
            \indent  \dbP\mbox{-}a.s. 
    \end{equation}
\end{condition}
    
To introduce the concept of transposition solution to \eqref{apHm} and present its well-posedness, firstly we should introduce the following test SEEs 
\begin{equation}\label{apHmt}
\left\{
\begin{aligned} 
& d\f = [A\f+ v_1]ds +  v_2 dW(s) \quad  \mbox{ in }(t,T],
\\
& \f(t)=\eta,
\end{aligned}
\right.
\end{equation}
where $t\in[0,T]$, $\eta\in L^{q}_{\cF_t}(\O;H)$, $v_1\in L^1_{\dbF}(t,T;L^{q}(\O;H))$, $v_2\in L^q_{\dbF}(t,T;H)$ with $p\in(1,2]$, $ {1}/{p} + {1}/{q} =1$.

\begin{definition}\label{definition1}
A pair of processes $(y(\cdot), Y(\cdot)) \in  D_{\dbF}([0,T];L^{p}(\O;H)) \times  L^p_{\dbF}(0,T;H)$ are defined to be transposition solutions to the Eq. \eqref{apHm}, if 
\begin{align*}
   & \dbE {\big\langle \f(T),y_T\big\rangle}_{H}
    - \dbE\int_t^T {\big\langle \f(s),f(s,y(s),Y(s) )\big\rangle}_Hds\\
    &\indent 
 = \dbE {\big\langle\eta,y(t)\big\rangle}_H
    + \dbE\int_t^T {\big\langle
    v_1(s),y(s)\big\rangle}_H ds + \dbE\int_t^T
    {\big\langle v_2(s),Y(s)\big\rangle}_H ds, 
\end{align*}
where $t\in [0,T]$,  $\eta\in L^{q}_{\cF_t}(\O;H)$,   $v_1(\cdot)\in L^1_{\dbF}(t,T;L^{q}(\O;H))$,  $v_2(\cdot)\in L^q_{\dbF}(t,T; H)$, and the test stochastic process $\f\in C_{\dbF}([t,T];L^{q}(\O;H))$ is the solution in the mild sense to \eqref{apHmt}. 
\end{definition}

The following lemma presents the well-posedness of $H$-valued BSEE \eqref{apHm}, whose proof can be found in detail in 
\cite[Theorem 4.19]{luqzhx21}. 
\begin{lemma} \label{lcdpq}
    Let Condition \ref{iclb} holds. Then for $f(\cdot ,0,0) \in L^{1}_{\dbF} (0,T;L^{p}(\O;H))$, $y_T \in L_{\cal F_T}^{p}( \O;H)$, there exists a unique transposition solution  $(y(\cdot),   Y(\cdot)) \in   D_{\dbF}([0,T]; L^{p} (\O;H)) \times  L^p_{\dbF} (0,T; H )$ to the Eq.  \eqref{apHm}. 
    Moreover, 
\begin{align}
 & |(y(\cdot), Y(\cdot))|_{D_{\dbF}([0,T]; L^{p} (\O;H)) \times  L^p_{\dbF} (0,T; H)}  \notag
    \\
 &\indent \le C(|y_T|_{L_{\cal F_T}^{p}( \O;H)} + |f(\cdot ,0,0)|_{L^{1}_{\dbF} (0,T;L^{p}(\O;H))}). 
\end{align}
\end{lemma}

Provided with the above preparations on SEEs and BSEEs, we present the FBSEEs 
\begin{equation}\label{aapm}
    \left\{
    \begin{aligned} 
    &d x(t)=(Ax(t)+f_1(t,x(t),y(t),Y(t)))dt + f_2(t,x(t),y(t),Y(t))dW(t) \quad \mbox{in}\  (0,T] 
        \\
    &dy(t) = - [ A^* y(t) - f_3(t,x(t),y(t),Y(t))]dt + Y(t) dW(t) \quad  \mbox{ in }[0,T),\\
    & x(0)=x_0, \quad  y(T) = y_T,
    \end{aligned}
    \right.
\end{equation}
where $f_i:[0,T]\times  \O \times H\times H \times H  \to H, i=1,2,3$ satisfy certain conditions. 
\begin{definition}
    The triple of adapted processes $(x,y,Y)$ are defined to be transposition solutions to the FBSEEs \eqref{aapm}, if $x$ is the mild solution to its forward part, and $(y,Y)$ are the transposition solutions to its backward part. 
\end{definition}
In the first section of this article, we review several papers in the literature on the current research status of FBSEEs. See \cite{AMP23,luqzhx21,xtm24,XXZ24} and the references therein for more details.

\section{Operators derived from state processes}\label{sbsd}
As preparations representing the variational expansion of the cost functional with respect to the control variable, from the well-posedness of the Eq. \eqref{zfxk}, for any initial pair $(t,\eta) \in [0,T)\times H$  and  $u(\cdot) \in L_{\mathbb{F}}^{2}(0, T ; U)$,  define the following operators: 
\begin{equation*}
\left\{
\begin{aligned}
&    \mathtt{M}_{t}: L_{\mathbb{F}}^{2}(t, T ; U) \rightarrow L_{\mathbb{F}}^{2}(t, T ; H), \\
&  (\mathtt{M}_{t} u(\cdot) )(\cdot) := x(\cdot ; t, 0, u,(0,0)), \quad  \forall u(\cdot) \in L_{\mathbb{F}}^{2}(0, T ; U) ;
\end{aligned}
\right.
\end{equation*}
\begin{equation*}
    \left\{
        \begin{aligned}
        &    \hat{\mathtt{M}}_{t}: L_{\mathbb{F}}^{2}(t, T ; U) \rightarrow L_{\mathcal{F}_T}^{2}(\O; H), \\
        &  \hat{\mathtt{M}}_{t} u(\cdot) := x(T ; t, 0, u,(0,0)), \quad  \forall u(\cdot) \in L_{\mathbb{F}}^{2}(0, T ; U) ;
        \end{aligned}
        \right.
\end{equation*}
\begin{equation*}
        \left\{
\begin{aligned}
    & \mathtt{N}_{t}: H \rightarrow L_{\mathbb{F}}^{2}(t, T ; H) , 
    \\
    &  (\mathtt{N}_{t} \eta )(\cdot) := x(\cdot ; t, \eta, 0,(0,0)),  \quad \forall \eta \in H; 
\end{aligned}
        \right.
\end{equation*}
\begin{equation*}
    \left\{
\begin{aligned}
& \hat{\mathtt{N}}_{t}: H \rightarrow L_{\mathcal{F}_T}^{2}(\O ; H) , 
\\
&  \hat{\mathtt{N}}_{t} \eta := x(T ; t, \eta, 0,(0,0)),  \quad \forall \eta \in H;
\end{aligned}
    \right.
\end{equation*}
and the processes 
\[\mathsf{h}_t(\cdot):=x(\cdot ; t, 0, 0,(b,\sigma )),\]
where $x(\cdot ;t, \eta, u,(b,\sigma ))$ solves \eqref{zfxk}. 

Besides, by the linear structure of the control system \eqref{zfxk}, it derives 
\begin{align}\label{cfjxb}
    x(s ;t, \eta, u,(b,\sigma ))= (\mathtt{M}_{t} u(\cdot) )(s) + (\mathtt{N}_{t} \eta )(s) + \mathsf{h}_t(s), \quad s\in[t,T], \mathbb{P}\mbox{-}a.s., t\in[0,T), 
\end{align}
and 
\begin{align}\label{cjxbz}
x(T ;t, \eta, u,(b,\sigma ))= \hat{\mathtt{M}}_{t} u(\cdot) + \hat{\mathtt{N}}_{t} \eta + \mathsf{h}_t(T), \quad \mathbb{P}\mbox{-}a.s.
\end{align}

To represent the variational expansion of the cost functional, we also should study the adjoint operators of the above operators $\mathtt{M}_{t}, \mathtt{N}_{t}, \hat{\mathtt{M}}_{t}, \hat{\mathtt{N}}_{t}, t\in[0,T)$ 
\begin{equation*}
    \left\{
\begin{aligned}
    & \mathtt{M}_{t}^{*}: L_{\mathbb{F}}^{2}(t, T ; H) \rightarrow L_{\mathbb{F}}^{2}(t, T ; U), 
    \\
    & \hat{\mathtt{M}}_{t}^{*}: L_{\mathcal{F}_{T}}^{2}(\Omega ; H) \rightarrow L_{\mathbb{F}}^{2}(t, T ; U), 
    \\
    & \mathtt{N}_{t}^{*}: L_{\mathbb{F}}^{2}(t, T ; H) \rightarrow H , 
    \\
    & \hat{\mathtt{N}}_{t}^{*}: L_{\mathcal{F}_{T}}^{2}(\Omega ; H) \rightarrow H. 
\end{aligned}
\right.
\end{equation*}
Consider the following BSEE 
\begin{equation}\label{dsbsc}
    \left\{
        \begin{aligned}
            & d y=- [(A+A_1)^{*} y+C^{*} Y+\xi ] d s+Y d W(s) \quad \text { in }[t, T), 
            \\
            & y(T)=y_{T},    
        \end{aligned}
        \right.
\end{equation}
where $ y_{T} \in L_{\mathcal{F}_{T}}^{2}(\Omega ; H) $  and $\xi(\cdot) \in L_{\mathbb{F}}^{2}(0, T ; H) $. 
There are several definitions of solutions to Eq. \eqref{dsbsc}. In this paper, we adopt the concept of transposition solution, which is particularly suited for handling optimal control problems. 
See \hyperref[definition1]{Definition \ref{definition1}} for details. 
Moreover, by \hyperref[lcdpq]{Lemma \ref{lcdpq}},  for any $y_{T} \in L_{\mathcal{F}_{T}}^{2}(\Omega ; H)$ and $\xi(\cdot) \in L_{\mathbb{F}}^{2}(0, T ; H)$, there exists a unique transposition solution 
\[(y(\cdot;y_T,\xi ), Y(\cdot;y_T,\xi)) \in D_{\mathbb{F}}([t, T] ; L^{2}(\Omega ; H)) \times L_{\mathbb{F}}^{2}(t, T ; H),\]
satisfying 
\[\sup _{t \leq s \leq T} \mathbb{E}|y(s;y_T,\xi)|_{H}^{2}+\mathbb{E} \int_{t}^{T}|Y(s;y_T,\xi)|_{H}^{2} d s \leq \mathcal{C} \mathbb{E} \big( |y_{T} |_{H}^{2}+\int_{t}^{T}|\xi(s)|_{H}^{2} d s \big).\]
This estimation guarantees the boundedness of the following four linear operators.

Now we present the following proposition about the adjoint operators.
\begin{proposition}\label{lsbbn}
    For any $y_{T} \in L_{\mathcal{F}_{T}}^{2}(\Omega ; H)$ and $\xi(\cdot) \in L_{\mathbb{F}}^{2}(t, T ; H)$,
\begin{equation} 
        \left\{
\begin{aligned}
         &(\mathtt{M}_{t}^{*} \xi )(s)=B^{*} y (s;0,\xi )+D^{*} Y (s;0,\xi ),  \text { a.e. } s \in[t, T] , 
        \\
        &\mathtt{N}_{t}^{*} \xi=y (t;0,\xi ), 
\\
     &(\hat{\mathtt{M}}_{t}^{*} y_T )(s)=B^{*} y (s;y_T,0)+D^{*} Y (s;y_T,0),  \text { a.e. } s \in[t, T] , 
    \\
    &\hat{\mathtt{N}}_{t}^{*} y_T=y (t;y_T,0), 
\end{aligned}
\right.
\end{equation}
where $(y,Y)$ is the transposition solution to \eqref{dsbsc}. 
\end{proposition}
\begin{proof}
We merely sketch the proof to illustrate the idea. 

Taking $f(\cdot, y,Y)=-A_1^* y-C^* Y -\xi $, the Eq. \eqref{dsbsc} can be regarded as in the form of \eqref{apHm}. 
Recalling the relation $L_{\mathbb F}^1(0,T;L^p(\O;H)) \subset L_{\mathbb F}^p(\O;L^1(0,T;H))$, under the assumptions in \eqref{H1}, the conditions in Lemma \ref{lcdpq} are fulfilled.

From the definition of transposition solutions to the Eq. \eqref{dsbsc}, 
taking the test processes $\varphi $ in \eqref{apHmt} as the controlled state processes $x$ in \eqref{zfxk} with $v_1=A_1 x+B u+b, v_2= Cx+Du+\sigma $, and in particular, $(b,\sigma )\equiv(0,0)$, 
it derives
\begin{align*}
\mathbb{E} \langle x(T;t,\eta ,u,(0,0)), y_{T} \rangle_{H}-\mathbb{E}\langle\eta, y(t)\rangle_{H} =\mathbb{E} \int_{t}^{T} ( \langle u(s), B^{*} y(s)+D^{*} Y(s) \rangle_{U}-\langle x(s), \xi(s)\rangle_{H} ) d s, 
\end{align*}
where $x$ is the mild solution to \eqref{zfxk} for any initial pair $(t,\eta) \in [0,T)\times H$  and  $u(\cdot) \in L_{\mathbb{F}}^{2}(0, T ; U)$.
Then making use of the operator representations for the processes $x$ in the beginning of this section, we obtain 
\begin{align}\label{dgdbs}
    &\mathbb{E} ( \langle\hat{\mathtt{N}}_{t} \eta+\hat{\mathtt{M}}_{t} u, y_{T} \rangle_{H}-\langle\eta, y(t)\rangle_{H} )
    \\
    &=\mathbb{E} \int_{t}^{T}\left( \langle u(s), B^{*} y(s)+D^{*} Y(s) \rangle_{U}- \langle (\mathtt{N}_{t} \eta )(s)+ (\mathtt{M}_{t} u )(s), \xi(s) \rangle_{H}\right) d s. \notag 
\end{align}
Moreover, taking  
\begin{equation*}
    \left\{
        \begin{aligned}
            \eta =0, 
            \\
            y_T=0, 
        \end{aligned}
        \right.        \quad 
    \left\{
        \begin{aligned}
            u(\cdot)=0, 
            \\
            y_T=0, 
        \end{aligned}
        \right.  
        \quad 
        \left\{
            \begin{aligned}
                \xi (\cdot)=0, 
                \\
                \eta =0, 
            \end{aligned}
            \right. 
            \quad 
            \left\{
                \begin{aligned}
                    u(\cdot)=0, 
                    \\
                    \xi (\cdot)=0, 
                \end{aligned}
                \right. 
\end{equation*}
by turns in the relations \eqref{dgdbs}, the results can be proved.
\end{proof}

\section{Main results for optimal control}\label{smos}

One method to study the open-loop solvability of the Problem (SLQ) is to characterize the coefficients of $\varepsilon^k$ in the variational representation of $J(u+\varepsilon v)$, $k\in  \mathbb N_+, \varepsilon \in \mathbb R, u,v\in L_{\mathbb{F}}^{2}(t, T ; U)$. In detail, we have 
\begin{theorem}
    Under conditions \eqref{H1} and \eqref{H2}, for any given initial pair $(t,\eta) \in [0,T)\times H$, the mapping $u(\cdot) \to \mathcal{J}(t,\eta,(b,\sigma ) ; u(\cdot))$ is Fr\'echet differentiable, and its derivative at $u(\cdot)$ is 
    \[{\cal D}_u \mathcal{J}(t,\eta,(b,\sigma ) ; u(\cdot)) = \Psi^{(t)}_1 u + \Psi^{(t)}_2 \eta + \varphi^{(t)}_2,\]
where the meanings of these operators are interpreted in \eqref{xsff}. 
\end{theorem}
\begin{proof}
    With the help of the \hyperref[lsbbn]{Proposition \ref{lsbbn}}, we deduce several representations which are essential to obtain the ultimate optimal necessary conditions.  
Taking the linear expressions for the state processes in \eqref{cfjxb} and \eqref{cjxbz} into the cost functional \eqref{bdjfs}, it can be verified that 
\begin{align*}
        \mathcal{J}(t,\eta,(b,\sigma ) ; u(\cdot))= & \frac{1}{2} \mathbb{E}\Big[ \int_{t}^{T} \big[ \langle Q(s) x(s), x(s)\rangle_{H}+\langle R(s) u(s), u(s)\rangle_{U} + 2 \langle S(s)x(s),u(s) \rangle_U 
        \\  & \indent 
        +2 \langle \mathfrak{q}(s), x(s)\rangle_{H} 
        +2 \langle \mathfrak{r} (s), u(s)\rangle_{U} \big] d s 
         +\langle G x(T), x(T)\rangle_{H} + 2 \langle \mathfrak{g}, x(T)\rangle_{H}\Big],  
\\  
 =&
    \frac{1}{2} \mathbb{E} \int_{t}^{T} \langle ((\mathtt{M}_{t}^*Q\mathtt{M}_{t} + \hat{\mathtt{M}}_{t}^*G\hat{\mathtt{M}}_{t} + S\mathtt{M}_{t} + \mathtt{M}_{t}^*S^* +R)u(\cdot))(s),u(s)\rangle_U ds
\\  &
+ \mathbb{E} \int_{t}^{T} \langle ((\mathtt{M}_{t}^*Q\mathtt{N}_{t} + \hat{\mathtt{M}}_{t}^*G\hat{\mathtt{N}}_{t} + S{\mathtt{N}}_{t} )\eta)(s) , u(s)\rangle_U ds
\\  &  
+\frac{1}{2} \langle (\mathtt{N}_{t}^*Q\mathtt{N}_{t} + \hat{\mathtt{N}}_{t}^*G\hat{\mathtt{N}}_{t} ) \eta , \eta \rangle_H  
\\   & 
+  \langle  \mathtt{N}_{t}^*(Q\mathsf{h}_{t}+\mathfrak{q}) + \hat{\mathtt{N}}_{t}^*(G\mathsf{h}_{t}(T)+ \mathfrak{g}) ,\eta \rangle_H 
\\  & 
+  \mathbb{E} \int_{t}^{T} \langle  (\mathtt{M}_{t}^*(Q\mathsf{h}_{t}+\mathfrak{q}) + \hat{\mathtt{M}}_{t}^*(G\mathsf{h}_{t}(T)+ \mathfrak{g}) + S\mathsf{h}_t + \mathfrak{r})(s), u(s)\rangle_U ds
\\  & 
+ \frac{1}{2} \mathbb{E} \big[ \langle G\mathsf{h}_{t}(T) +2 \mathfrak{g},\mathsf{h}_{t}(T)\rangle_H  
+ \int_{t}^{T} \langle (Q\mathsf{h}_{t}+2 \mathfrak{q})(s), \mathsf{h}_{t} (s)\rangle_H ds \big]. 
\end{align*}
To simplify the notations, denote 
\begin{align}\label{xsff}
&\Psi^{(t)}_1=\mathtt{M}_{t}^*Q\mathtt{M}_{t} + \hat{\mathtt{M}}_{t}^*G\hat{\mathtt{M}}_{t} + S\mathtt{M}_{t} + \mathtt{M}_{t}^*S^* +R , \notag 
\\
&
\Psi^{(t)}_2= \mathtt{M}_{t}^*Q\mathtt{N}_{t} + \hat{\mathtt{M}}_{t}^*G\hat{\mathtt{N}}_{t} + S{\mathtt{N}}_{t} , \notag 
\\
&
\Psi^{(t)}_3=\mathtt{N}_{t}^*Q\mathtt{N}_{t} + \hat{\mathtt{N}}_{t}^*G\hat{\mathtt{N}}_{t},
\\
&
\varphi^{(t)}_1=\mathtt{N}_{t}^*(Q\mathsf{h}_{t}+\mathfrak{q}) + \hat{\mathtt{N}}_{t}^*(G\mathsf{h}_{t}(T)+ \mathfrak{g}), \notag 
\\
&
\varphi^{(t)}_2= \mathtt{M}_{t}^*(Q\mathsf{h}_{t}+\mathfrak{q}) + \hat{\mathtt{M}}_{t}^*(G\mathsf{h}_{t}(T)+ \mathfrak{g}) + S\mathsf{h}_t + \mathfrak{r}, \notag 
\\
&
\varphi^{(t)}_3=\frac{1}{2} \mathbb{E} \big[ \langle G\mathsf{h}_{t}(T) +2 \mathfrak{g},\mathsf{h}_{t}(T)\rangle_H  
+ \int_{t}^{T} \langle (Q\mathsf{h}_{t}+2 \mathfrak{q})(s), \mathsf{h}_{t} (s)\rangle_H ds \big].  \notag 
\end{align} 
From the regularities of the coefficients and estimates for the operators $\mathtt{M}_{t}, \hat{\mathtt{M}}_{t}, \mathtt{N}_{t}, \hat{\mathtt{N}}_{t}$ as well as their adjoint operators, it can be derived that   $\Psi^{(t)}_1$ is a bounded linear operator from $L_{\mathbb{F}}^{2}(t, T ; U)$ to itself. $\Psi^{(t)}_2$ is a bounded linear operator from $H$ to $L_{\mathbb{F}}^{2}(t, T ; U)$. $\Psi^{(t)}_3$ is a bounded linear operator from $H$ to $H$. 

Then 
\begin{align}\label{zdzj}
    \mathcal{J}(t,\eta,(b,\sigma ) ; u(\cdot))=& \frac{1}{2} \mathbb{E} \int_{t}^{T} \langle (\Psi^{(t)}_1 u(\cdot))(s),u(s)\rangle_U ds + \mathbb{E} \int_{t}^{T} \langle (\Psi^{(t)}_2 \eta )(s),u(s)\rangle_U ds  + \frac{1}{2} \langle \Psi^{(t)}_3 \eta , \eta \rangle_H 
    \\
    & + \langle \varphi^{(t)}_1, \eta \rangle_H +    \mathbb{E} \int_{t}^{T}\langle  \varphi^{(t)}_2(s), u(s)\rangle_U ds + \varphi^{(t)}_3. \notag 
\end{align}
Moreover, for any $\varepsilon \in \mathbb R$ and $v\in  L_{\mathbb{F}}^{2}(t, T ; U)$, 
\begin{align}\label{dzyb}
    \mathcal{J}(t,\eta,(b,\sigma ) ; u+\varepsilon v)=& \frac{1}{2} \mathbb{E} \int_{t}^{T} \langle (\Psi^{(t)}_1 (u+\varepsilon v) )(s),(u+\varepsilon v)(s)\rangle_U ds 
+ \frac{1}{2} \langle \Psi^{(t)}_3 \eta , \eta \rangle_H + \langle \varphi^{(t)}_1, \eta \rangle_H   \notag 
\\
&  
    + \mathbb{E} \int_{t}^{T} \langle (\Psi^{(t)}_2 \eta )(s),(u+\varepsilon v)(s)\rangle_U ds 
    +    \mathbb{E} \int_{t}^{T}\langle  \varphi^{(t)}_2(s), (u+\varepsilon v)(s)\rangle_U ds + \varphi^{(t)}_3   \notag 
    \\
    =& \mathcal{J}(t,\eta,(b,\sigma ) ; u) + \varepsilon^2 \mathcal{J}(t,0,(0,0);v) 
    \\  & + \varepsilon \mathbb{E} \int_{t}^{T} \langle (\Psi^{(t)}_1 u + \Psi^{(t)}_2 \eta + \varphi^{(t)}_2)(s), v(s)\rangle_U ds.  \notag  
\end{align}
Then by the definition of the Fr\'echet derivative, the conclusion follows.
\end{proof}
Actually, due to the relation \eqref{zdzj}, we have 
\begin{corollary}
    1) One of the necessary conditions for the problem (SLQ) associated with the initial pair $(t,\eta)$ to be finite is $\Psi^{(t)}_1\ge0$. 

    2) The problem (SLQ) with the initial pair $(t,\eta)$ is solvable if and only if $\Psi^{(t)}_1\ge0$ and there exists $ u^*(\cdot)\in L_{\mathbb{F}}^{2}(t, T ; U)$, such that ${\cal D}_u \mathcal{J}(t,\eta,(b,\sigma ) ; u^*(\cdot)) = \Psi^{(t)}_1  u^* + \Psi^{(t)}_2 \eta + \varphi^{(t)}_2=0$.

    3) 
    Moreover, if $\Psi^{(t)}_1 \gg  0$, the cost functional $\mathcal{J}(t,\eta,(b,\sigma ) ; u(\cdot))$ uniquely has an explicit minimizer $ u^* =- ({\Psi^{(t)}_1})^{-1}(\Psi^{(t)}_2 \eta + \varphi^{(t)}_2)$. 
\end{corollary}

Provided with the above preparation, we present the necessary and sufficient conditions for open-loop optimal controls. 
\begin{theorem}\label{dbtz}
    Under conditions \eqref{H1} and \eqref{H2}, for any given initial pair $(t,\eta) \in [0,T)\times H$, $u(\cdot) \in  L_{\mathbb{F}}^{2}(t, T ; U)$ is an open-loop optimal control to the Problem (SLQ) associated with $(t,\eta)$ if and only if 
    
    1) the mapping $u(\cdot) \to \mathcal{J}(t,0,(0,0) ; u(\cdot))$ is convex; 

  2)  the process $u(\cdot)$ satisfies 
\[B^*y + D^*Y + Sx+Ru + \mathfrak{r}=0, \quad  a.e. (s,\o)\in(t,T)\times \O,\]
  where the triple $(x,y,Y)$ satisfy the following 
  \begin{equation}\label{xyz1}
      \left\{
      \begin{aligned}
  & d x=[(A+A_1) x+Bu+b] d s +[Cx+Du+\sigma] d W(s) \quad \text{in}\ (t,T],
  \\ 
  & d y=- [(A+A_1)^{*} y+C^{*} Y+ Qx+S^*u+\mathfrak{q}] d s+Y d W(s) \quad \text { in }[t, T), 
  \\
  & x(t)=\eta, \quad  y(T)=Qx_{T}+\mathfrak{g}.
  \end{aligned}
  \right.
  \end{equation}
\end{theorem}
\begin{proof}
By the definition of related operators and \hyperref[lsbbn]{Proposition \ref{lsbbn}}, it can be calculated that 
\begin{align*}
\Psi^{(t)}_1 u & = \mathtt{M}_{t}^*(Q\mathtt{M}_{t}+S^*) u + \hat{\mathtt{M}}_{t}^*G\hat{\mathtt{M}}_{t} u + S\mathtt{M}_{t}u  +Ru 
\\
&
= B^*y(\cdot;0,Qx(\cdot;t,0,u,(0,0))+S^*u)+ D^*Y(\cdot;0,Qx(\cdot;t,0,u,(0,0))+S^*u)
\\
& \indent 
+ B^*y(\cdot;Gx(T;t,0,u,(0,0)),0)+ D^*Y(\cdot;Gx(T;t,0,u,(0,0)) ,0)
\\
& \indent 
+ Sx(\cdot;t,0,u,(0,0)) +Ru, 
\\ 
\Psi^{(t)}_2 \eta &= \mathtt{M}_{t}^*Q\mathtt{N}_{t} \eta + \hat{\mathtt{M}}_{t}^*G\hat{\mathtt{N}}_{t} \eta + S{\mathtt{N}}_{t}  \eta 
\\  &  
= B^*y(\cdot;0,Qx(\cdot;t,\eta,0,(0,0)))+ D^*Y(\cdot;0,Qx(\cdot;t,\eta,0,(0,0)))
\\  &  \indent 
+ B^*y(\cdot;Gx(T;t,\eta,0,(0,0)),0)+ D^*Y(\cdot;Gx(T;t,\eta,0,(0,0)),0)
\\  &  \indent 
+ Sx(\cdot;t,\eta,0,(0,0)) , 
\\ 
\varphi^{(t)}_2 & = \mathtt{M}_{t}^*(Q\mathsf{h}_{t}+\mathfrak{q}) + \hat{\mathtt{M}}_{t}^*(G\mathsf{h}_{t}(T)+ \mathfrak{g}) + S\mathsf{h}_t + \mathfrak{r}
\\  &
= B^*y(\cdot;0,Q\mathsf{h}_{t}+\mathfrak{q}) + D^*Y(\cdot;0,Q\mathsf{h}_{t}+\mathfrak{q})
\\  &  \indent 
+ B^*y(\cdot;G\mathsf{h}_{t}(T)+\mathfrak{g},0) + D^*Y(\cdot;G\mathsf{h}_{t}(T)+\mathfrak{g},0) 
+ S\mathsf{h}_t + \mathfrak{r}. 
\end{align*}
Sum it up resulting in 
\begin{align*}
\Psi^{(t)}_1 u + \Psi^{(t)}_2 \eta + \varphi^{(t)}_2 & = 
B^*y(\cdot;Gx(T;t,\eta,u,(b,\sigma))+\mathfrak{g},Qx(\cdot;t,\eta,u,(b,\sigma))+S^*u+\mathfrak{q})
\\  &  \indent 
+ D^*Y(\cdot;Gx(T;t,\eta,u,(b,\sigma))+\mathfrak{g},Qx(\cdot;t,\eta,u,(b,\sigma))+S^*u+\mathfrak{q})
\\  &  \indent 
+ Sx(\cdot;t,\eta,u,(b,\sigma)) +Ru + \mathfrak{r}.
\end{align*}
Utilizing the linear structure in \eqref{cfjxb}, \eqref{cjxbz} and \eqref{dsbsc}, it can be deduced that 
\begin{align}\label{xdsb}
\Psi^{(t)}_1 u + \Psi^{(t)}_2 \eta + \varphi^{(t)}_2 = B^*y + D^*Y + Sx+Ru + \mathfrak{r}.
\end{align}

The ``sufficient" part: 
By \eqref{dzyb}, it indicates that for any $\varepsilon \in \mathbb R$ and $u, v\in  L_{\mathbb{F}}^{2}(t, T ; U)$, 
\begin{align}\label{xzef}
& \mathcal{J}(t,\eta,(b,\sigma ) ; u+\varepsilon v) -  \mathcal{J}(t,\eta,(b,\sigma ) ; u) 
\\  & \indent 
=  \varepsilon^2 \mathcal{J}(t,0,(0,0);v) + \varepsilon \mathbb{E} \int_{t}^{T} \langle (\Psi^{(t)}_1 u + \Psi^{(t)}_2 \eta + \varphi^{(t)}_2)(s), v(s)\rangle_U ds. \notag 
\end{align}
Then by \eqref{xdsb} and \eqref{xzef}, $u$ is an optimal control to the Problem (SLQ) associated with $(t,\eta)$ if 1) and 2) hold. 

The ``necessary'' part: 
By \eqref{xzef}, if $u$ is an optimal control, for any $\varepsilon \in \mathbb R$ and $ v\in  L_{\mathbb{F}}^{2}(t, T ; U)$, we have 
\begin{align*} 
    \psi (\varepsilon ,v) & := 
\varepsilon^2 \mathcal{J}(t,0,(0,0);v) + \varepsilon \mathbb{E} \int_{t}^{T} \langle (B^*y + D^*Y + Sx+Ru + \mathfrak{r})(s), v(s)\rangle_U ds 
\\
&
\  \ge 0. 
\end{align*}
It indicates that $\psi (\cdot ,v)$ is quadratic and nonnegative in $\varepsilon \in \mathbb R$ for any fixed $ v\in  L_{\mathbb{F}}^{2}(t, T ; U)$. By elementary calculus for quadratic functions, we obtain 
\begin{equation*}
    \left\{
    \begin{aligned}
        & \mathcal{J}(t,0,(0,0);v) \ge 0, \\
& \mathbb{E} \int_{t}^{T} \langle (B^*y + D^*Y + Sx+Ru + \mathfrak{r})(s), v(s)\rangle_U ds =0.
    \end{aligned}
    \right.
\end{equation*}
Moreover, from the bi-linearity of the cost funtional $\mathcal{J}(t,0,(0,0);v)$, and the equivalence between the convexity and the nonnegative for a  bi-linearity function, it derives that 1) holds. Besides, by the arbitrariness of $v\in L_{\mathbb{F}}^{2}(t, T ; U)$, the pointwise condition 2) holds. 
\end{proof}

\begin{remark}
    If the coefficients of ${\langle Sx,u\rangle}_U $ in the cost functional is $1$ but not $2$, the term $$B^*y(\cdot;0,Qx(\cdot;t,0,u,(0,0))+S^*u)+ D^*Y(\cdot;0,Qx(\cdot;t,0,u,(0,0))+S^*u)$$ in $\Psi^{(t)}_1 u$ can be changed to $B^*y(\cdot;0,Qx(\cdot;t,0,u,(0,0)))+ D^*Y(\cdot;0,Qx(\cdot;t,0,u,(0,0)))$, and consequently the $S^*u$ in the second line of \eqref{xyz1} should be erased. 
    It means that the above difference in coefficients does not have essential technical impact. 
\end{remark}

The system \eqref{xyz1} seems to be decoupled at first glance. However, it should be understood in the sense of 
\begin{equation}\label{xyzds}
    \left\{
    \begin{aligned}
& d x=[(A+A_1) x+Bu+b] d s +[Cx+Du+\sigma] d W(s) \quad \text{in}\ (t,T],
\\ 
& d y=- ((A+A_1)^{*} y+C^{*} Y+ Qx+S^*u+\mathfrak{q}) d s+Y d W(s) \quad \text { in }[t, T), 
\\
&
B^*y + D^*Y + Sx+Ru + \mathfrak{r} = 0, \quad a.e. (s,\o)\in(t,T)\times \O, 
\\
& x(t)=\eta, \quad  y(T)=Qx_{T}+\mathfrak{g}, 
\end{aligned}
\right.
\end{equation}
which is a fully-coupled forward and backward stochastic evolution system, and is rather challenging to solve (cf. \cite{xtm24} and the references therein). 
The coupling comes from the linear stationary condition, which is also called feedback constrained condition. In other word, if $u$ is an open-loop optimal control for the Problem (SLQ) with the initial pair $(t,\eta)$, it must be determined by the system of Eq. \eqref{xyzds}.
Actually, \eqref{xyzds} is in general called the optimality system of Problem (SLQ) in literature.

\section{Linear quadratic stochastic differential games}\label{sygs}

As an application of the conclusions on stochastic optimal control problems in the previous section, we study the following models on LQ two-person  stochastic differential games. 

For any given initial pair $(t,\eta) \in [0,T)\times H$, consider the following controlled SEEs: 
\begin{equation}\label{zfxg}
    \left\{
    \begin{aligned}
& d x(s)=[(A+A_1(s)) x(s)+B_1(s) u_1(s)+B_2(s) u_2(s)+b(s)] d s 
\\
  & \indent\indent  +[C(s) x(s)+D_1(s) u_1(s)+D_2(s) u_2(s)+\sigma (s)] d W(s) \quad \text{in}\ (t,T],
\\
& x(t)=\eta,      
\end{aligned}
\right.
\end{equation}
while the cost functional for player $i$ is defined with $i=1,2$ by
\begin{equation}\label{bdf}
    \begin{aligned}
        \mathcal{J}^i(t,\eta ; u_1,u_2)=\frac{1}{2} \mathbb{E}\big\{ & \int_{t}^{T}\big[\langle Q^i(s) x(s), x(s)\rangle_{H} + 2\langle S_1^i(s) x(s), u_1(s)\rangle_{U}  + 2\langle S_2^i(s) x(s), u_2(s)\rangle_{U}   
        \\  
        & \indent +\langle R^i_{11}(s) u_1(s), u_1(s)\rangle_{U}+\langle R^i_{12}(s) u_2(s), u_1(s)\rangle_{U}+\langle R^i_{21}(s) u_1(s), u_2(s)\rangle_{U}
        \\ 
 &  \indent  +\langle R^i_{22}(s) u_2(s), u_2(s)\rangle_{U} 
 +2 \langle \mathfrak{r}_1 (s), u_1(s)\rangle_{U}+2 \langle \mathfrak{r}_2 (s), u_2(s)\rangle_{U} 
 \\  
 & \indent  +2 \langle \mathfrak{q}^i(s), x(s)\rangle_{H} \big] d s 
+\langle G^i x(T), x(T)\rangle_{H}  +2\langle  \mathfrak{g}^i, x(T)\rangle_{H} \big\}. 
    \end{aligned}
\end{equation}
Formally, the cost functionals can be presented in the more explicit form
\begin{equation*}\label{gcfs}
    \begin{aligned}
      \mathcal{J}^i(t,\eta ; u_1,u_2)  = & \frac{1}{2} \mathbb{E}\bigg[ \int_{t}^{T} \left\langle 
            \begin{pmatrix}
                Q^i(s) & {S_1^i}^*(s)& {S_2^i}^*(s)  \\
                {S_1^i}(s) & R^i_{11}(s)  & R^i_{12}(s)  \\
                {S_2^i}(s) &  R^i_{21}(s) & R^i_{22}(s) 
            \end{pmatrix}
            \begin{pmatrix}
                x(s) \\  u_1(s) \\ u_2(s)
            \end{pmatrix} + 2 \begin{pmatrix}\mathfrak{q}^i(s) \\  \mathfrak{r}^i_1 (s) \\  \mathfrak{r}^i_2 (s)  \end{pmatrix},
            \begin{pmatrix}
                x(s) \\  u_1(s) \\ u_2(s)
            \end{pmatrix}
                 \right\rangle  d s 
      \\  & \indent    +\langle G^i x(T) + 2 \mathfrak{g}^i, x(T)\rangle_{H}\bigg].  
    \end{aligned}
\end{equation*}

The LQ stochastic two-person differential games are  formulated as follows. 

\noindent
\textbf{Problem (SDG).} 
For any initial pair $(t,\eta) \in [0,T)\times H$, what the controls $u_i(\cdot) \in L_{\mathbb{F}}^{2}(t, T ; U)$ should the players choose to  minimize their payoff $\mathcal{J}^i(t,\eta ; u_1,u_2)$?

\begin{definition}
A pair of strategies $ (u_{1}^{*}, u_{2}^{*} ) \in L_{\mathbb{F}}^{2}(t, T ; U) \times L_{\mathbb{F}}^{2}(t, T ; U)$ are called an open-loop Nash equilibrium of the Problem (SDG) with the initial pair $(t, \eta )$, if
\[\mathcal{J}^{1}\left(t, \eta  ; u_{1}^{*}, u_{2}^{*}\right) \leqslant \mathcal{J}^{1}\left(t, \eta ; u_{1}, u_{2}^{*}\right), \quad  \forall u_{1} \in L_{\mathbb{F}}^{2}(t, T ; U), \]
\[\mathcal{J}^{2}\left(t, \eta ; u_{1}^{*}, u_{2}^{*}\right) \leqslant \mathcal{J}^{2}\left(t, \eta ; u_{1}^{*}, u_{2}\right), \quad  \forall u_{2} \in L_{\mathbb{F}}^{2}(t, T ; U).\]
\end{definition}

As commonly done in literature, we study the open-loop two-person stochastic differential game by investigating two related stochastic optimal control problems as follows.
Suppose that $(u_1^*, u_2^*)$ is an open-loop Nash equilibrium of Problem (SDG) for the initial pair $(t , \eta )$, then consider the following two  stochastic optimal control problems: 

\noindent \textbf{Problem (SDG1)}  \ 
To minimize 
\begin{align}
    \mathcal{J}(t, \eta ; u_1):=  \mathcal{J}^1(t, \eta ; u_1,u_2^*). 
\end{align}
subject to the state equation
\begin{equation}\label{zfxg1}
    \left\{
    \begin{aligned}
& d x(s)=[(A+A_1(s)) x(s)+B_1(s) u_1(s)+B_2(s) u^*_2(s)+b(s)] d s 
\\
  & \indent\indent  +[C(s) x(s)+D_1(s) u_1(s)+D_2(s) u^*_2(s)+\sigma (s)] d W(s) \quad \text{in}\ (t,T],
\\
& x(t)=\eta; 
\end{aligned}
\right.
\end{equation}

\noindent \textbf{Problem (SDG2)}  \ 
To minimize 
\begin{align}
    \mathcal{J}(t, \eta ; u_2):=  \mathcal{J}^2(t, \eta ; u_1^*,u_2). 
\end{align}
subject to the state equation 
\begin{equation}\label{zfxg2}
    \left\{
    \begin{aligned}
& d x(s)=[(A+A_1(s)) x(s)+B_1(s) u^*_1(s)+B_2(s) u_2(s)+b(s)] d s 
\\
  & \indent\indent  +[C(s) x(s)+D_1(s) u^*_1(s)+D_2(s) u_2(s)+\sigma (s)] d W(s) \quad \text{in}\ (t,T],
\\
& x(t)=\eta. 
\end{aligned}
\right.
\end{equation}

Since $(u_1^*, u_2^*)$ is an open-loop Nash equilibrium of Problem (SDG) associated with the initial pair $(t , \eta )$, it deduces that $u_1^*$ is an open-loop optimal control of Problem (SDG1), while $u_2^*$ is an open-loop optimal control of Problem (SDG2). To apply the \hyperref[dbtz]{Theorem \ref{dbtz}}, we impose the following conditions supposed for the coefficients and weighting operators. 
\begin{condition}
Assume that the coefficients in \eqref{zfxg} and weighting operators in \eqref{bdf} satisfy the following conditions 
\begin{equation}\tag{GH1}\label{GH1}
        \begin{aligned}
            &   A_1(\cdot) \in L_{\mathbb{F}}^{1} (0, T ; L^\infty(\O;\mathcal{L}(H))), \quad  B_i(\cdot) \in L_{\mathbb{F}}^\infty(\O; L^{2} (0, T ; \mathcal{L}(U ; H))), 
               \\
            &   C(\cdot) \in L_{\mathbb{F}}^{2} (0, T ; L^\infty(\O;\mathcal{L}(H))),  \quad  D_i(\cdot) \in L_{\mathbb{F}}^{\infty} (0, T ; \mathcal{L} (U ;H ) ),
            \\
            &  b(\cdot) \in L_{\mathbb F}^2 (\O; L^1 (0,T; H)),  \quad   \sigma(\cdot)\in L_{\mathbb F}^2 (0,T; H),  \quad  i=1,2,
        \end{aligned}
\end{equation}
and 
\begin{equation}\label{GH2}\tag{GH2}
        \begin{aligned}
        & Q^i(\cdot) \in L_{\mathbb{F}}^\infty(\O; L^{2}(0, T ; \mathcal{S}(H))),  \quad  S_1^i(\cdot),S_2^i(\cdot) \in L_{\mathbb{F}}^\infty(\O; L^{2} (0, T ; \mathcal{L}(H; U))), 
        \\  & 
        R^i_{11}(\cdot),R^i_{12}(\cdot),R^i_{21}(\cdot),R^i_{22}(\cdot) \in L_{\mathbb{F}}^{\infty}(0, T ; \mathcal{S}(U)),  \quad  G^i \in  L_{\mathcal{F}_T}^{\infty}(\O;\mathcal{S}(H)), 
        \\   &  
        {\mathfrak q}^i \in  L_{\mathbb F}^2(\O;L^1(t,T;H)), \quad  {\mathfrak r}_1^i,{\mathfrak r}_2^i  \in L_{\mathbb{F}}^{2}(t, T ; U),   \quad  {\mathfrak g}^i \in L_{\mathcal{F}_T}^{2}(\O; H), \quad  i=1,2.
    \end{aligned}
\end{equation}
\end{condition}
Now we present the main conclusions of this section. 
\begin{theorem}
Under conditions \eqref{GH1} and \eqref{GH2},  $ (u_{1}^{*}, u_{2}^{*} ) $ is an open-loop Nash equilibrium of Problem (SDG) for the initial pair $(t, \eta )$ if and only if the following conditions for $i=1,2$ hold:

\noindent 1)  $ (u_{1}^{*}, u_{2}^{*} ) $ satisfy the feedback constrained conditions 
\[B_i^*y_i + D_i^*Y_i + S^i_i x+R^i_{ii}u^*_i + \frac{1}{2} ({R^i_{21}} + {R^i_{12}})^* u^*_{3-i} + \mathfrak{r}_i = 0, \quad a.e. (s,\o)\in(t,T)\times \O, \]
where the triple $(x,y_i,Y_i)$ are transposition solutions to the following FBSEEs 
\begin{equation*}
        \left\{
        \begin{aligned}
    & d x=[(A+A_1) x+B_1u_1^*+B_2u_2^*+b] d s +[Cx+D_1u_1^*+D_2u_2^*+\sigma] d W(s) \quad \text{in}\ (t,T],
    \\ 
    & d y_i=- [(A+A_1)^{*} y_i+C^{*} Y_i+ Q^i x+{S_1^i}^*u^*_1+{S_{2}^i}^*u^*_{2}+\mathfrak{q}^i] d s+Y_i d W(s) \quad \text { in }[t, T), 
    \\
    & x(t)=\eta, \quad  y_i(T)=Q^i x_{T}+\mathfrak{g}^i,  
    \end{aligned}
    \right.
\end{equation*}

\noindent 2) the map $u_i\to \mathcal{J}_0^i(t,0 ; u_i)$ is convex, where $\mathcal{J}_0^i(t,0 ; u_i)$ is the cost functional to the homogeneous optimal control problems with controlled state equations 
\begin{equation*}
    \left\{
    \begin{aligned}
& d x_i=[(A+A_1) x_i+B_iu_i ] d s +(Cx_i+D_iu_i ) d W(s) \quad \text{in}\ (t,T],  
\\ 
& x_i(t)=0, 
\end{aligned}
\right.
\end{equation*}
and cost functional
\begin{equation*}
    \begin{aligned}
        \mathcal{J}_0^i(t,0 ; u_i)=\frac{1}{2} \mathbb{E}\big[ & \int_{t}^{T}\big(\langle Q^i(s) x_i(s), x_i(s)\rangle_{H} +\langle R^i_{ii}(s) u_i(s), u_i(s)\rangle_{U} + 2{\langle S_i^i x_i , u_i\rangle}_U  \big) d s 
        \\  & 
+\langle G^i x_i(T), x_i(T)\rangle_{H}   \big]. 
    \end{aligned}
\end{equation*}
\end{theorem}
\begin{proof}
    It can be verified by lengthy but straight calculus with the help of \hyperref[dbtz]{Theorem \ref{dbtz}}, and we omit the details. 
\end{proof}
Note that in the above claim 2), the convexity of the map $u_i\to \mathcal{J}_0^i(t,0 ; u_i)$ is equivalent to $\mathcal{J}_0^i(t,0 ; u_i)\ge0$ for all $u_i\in  L_{\mathbb{F}}^{2}(t, T ; U)$.

\section{Conclusions and discussions}\label{slcd}
We present necessary and sufficient conditions for open-loop optimal controls for linear quadratic stochastic optimal control problems in infinite dimension without Markovian restriction for coefficients, and characterize the Fr\'echet derivatives of the cost functional with respect to the control variable, which are exactly the stationary conditions. As applications, we employ the results to study open-loop Nash equilibria for two-person stochastic differential games. Since the examples in Lü et al.\cite{LWZ17} show that solvable stochastic LQ problems may NOT have feedback control, open-loop optimal controls are sometimes the only choice to identify the optimal controls. Besides, by Sun-Yong \cite{sunyong19b1}, the study on feedback optimal control for LQ problems relies on those results on open-loop cases. As a result, our study is a necessary and important step for LQ stochastic optimal control problems in infinite dimensional setting. Our study on the optimal feedback operators for LQ stochastic optimal control problems in infinite dimensional and nonhomogeneous setting will be presented elsewhere.

\section*{Competing interests}
The authors have no competing interests to declare that are relevant to the content of this article.  

\section*{Funding}
Research supported by National Key R\&D Program of China (No. 2022YFA1006300) and the NSFC (No. 12271030). 

\section*{Data availability}
No data was used for the research described in the article.

\bibliographystyle{IEEEtran}

\end{document}